\documentclass[pdftex, 10pt]{article}
\usepackage{graphicx}
\usepackage{tikz}
\usetikzlibrary{patterns}
\usetikzlibrary{intersections, calc, arrows.meta, patterns}
\usepackage{times}
\usepackage{amsmath}
\usepackage{hyperref}
\usepackage{amssymb}
\usepackage{amsthm}
\usepackage{amsfonts}
\usepackage{multicol}
\usepackage{cases}
\usepackage{bm}
\usepackage{fancybox,ascmac}
\usepackage{amsfonts}
\usepackage{latexsym}
\usepackage{mathcomp}
\usepackage{comment}
\usepackage{mathrsfs}
\usepackage{mathtools} 
\usepackage{mathtools,amssymb,amsthm}
\usepackage{enumitem}
\usepackage{accents}
\usepackage{float}
\usepackage{graphicx}
\usepackage{tkz-euclide}
\tikzset{elegant/.style={smooth,thick,samples=50,cyan}}
\usepackage{ascmac}
\usepackage{mathrsfs}
\usepackage{url}
\usepackage{color}
\usepackage[T1]{fontenc}
\usepackage[final]{microtype}
\usepackage{pgfplots}
\usepackage{titlesec}
\usepackage{bm}
\titlelabel{\thetitle.\quad}
\mathtoolsset{showonlyrefs}

\DeclareMathOperator{\supp}{supp}
\newcommand{\nc}{\newcommand}
\nc{\pr}{\partial}
\nc{\bR}{\mathbb{R}}
\nc{\bN}{\mathbb{N}}
\nc{\dx}{\frac{d}{dx}}
\nc{\dxx}{\frac{d^2}{dx^2}}
\nc{\prx}{\frac{\pr}{\pr x}}
\nc{\prt}{\pr_t}
\nc{\prxx}{\frac{\pr^2}{\pr x^2}}
\nc{\prtt}{\pr^2_t}
\nc{\ga}{\gamma}
\nc{\Hg}{{\dot{H}^{-\ga}}}
\nc{\eps}{\epsilon}
\nc{\lag}{\langle}
\nc{\rag}{\rangle}
\nc{\la}{\lambda}
\nc{\what}{\widehat}
\nc{\les}{\lesssim}
\nc{\gtr}{\gtrsim}
\nc{\cN}{\mathcal{N}}
\nc{\cM}{\mathcal{M}}
\nc{\al}{\alpha}
\nc{\nn}{\nonumber}
\nc{\cnl}{\chi_{|\nabla|<\eps}}
\nc{\cng}{\chi_{|\nabla|\ge\eps}}
\nc{\cxil}{\chi_{|\xi|<\eps}}
\nc{\cxig}{\chi_{|\xi|\ge\eps}}
\nc{\exi}{\exists}
\nc{\tr}{\textrm}
\nc{\tref}[1]{\textnormal{\ref{#1}}}
\nc{\coq}{\coloneqq}
\begin{document}

	\title{On semilinear damped wave equations with initial data in homogeneous Sobolev spaces}
	\author{Mitsuhiro MATSUNAGA}
	\maketitle
	\newtheorem{thm}{Theorem}
	\newtheorem{lem}{Lemma}
	\theoremstyle{remark}
	\newtheorem{rem}{Remark}
	\begin{abstract}
		  In this paper, we study semilinear damped equations $u_{tt}+u_t-\Delta u=|u|^p$ with the initial data in $(\Hg\cap H^s)\times(\Hg\cap L^2)$ with the dimension $n\le 6$. Chen-Reissig \cite{chenreissig2023} studied the case $0<\ga\le\min\{n/2, (-n+\sqrt{n^2+16n})/4\}$ and showed that the exponent $p_{\mathrm{crit}}=1+4/(n+2\ga)$ of $p$ distinguishes  the time global existence and the blow-up of solution. In this paper, we discuss the case $\ga\ge\min\{n/2, (-n+\sqrt{n^2+16n})/4\}$ and show that the critical exponent is not $1+4/(n+2\ga)$ but
		  \begin{align}
		  \begin{cases}
		  1+2/n & \text{if }n=1\text{ or }2\\
		  \frac{n+\sqrt{n^2+16n}}{2n} & \text{if }3\le n\le 6.
		  \end{cases}
		  \end{align}
	\end{abstract}
	\section{Introduction}
	In this paper, we study the following Cauchy problem for the semilinear damped wave equation:
	\begin{align}
		\begin{cases}
			\prtt u+\prt u-\Delta u=|u|^p&x\in\bR^n, t>0,\\
			\left(u(0,x),\prt u(t,x)|_{t=0}\right)=(\eps u_0,\eps u_1)&x\in\bR^n,
		\end{cases}\label{bdw}
	\end{align}
	with $p>1$ and $0<\eps\ll1$. In particular, we study the case when $(u_0,u_1)\in(\Hg\cap H^s)\times(\Hg\cap L^2), \ga>0$ and $s\in(0,1]$. Here,
	\begin{align}
		\|f\|_{\Hg}:=\left\||\xi|^{-\ga}\hat{f}\right\|_{L^2}
	\end{align}
	and
	\begin{align}
		\Hg:=\left\{f\in\mathcal{S}'\ \middle|\ \hat{f}\in L^1_{loc},\ \|f\|_{\Hg}<\infty\right\}.
	\end{align}
	Our main purpose is to determine the critical exponent $p_\mathrm{c}$ of the power $p$ in \eqref{bdw} which distinguishes the global existence of solutions and their blow-up, and the lifespan $T$ when blow-up.

With initial data $(u_0,u_1)\in (L^1\cap H^s)\times (L^1\cap L^2)$, many papers (for example, \cite{Hosono2004,ikeda2015note,matsumura1976,narazaki2004,nishihara2003,ono2003,todorova2001critical}) show that the critical exponent of \eqref{bdw} is the Fujita exponent $p_{\mathrm{F}} := 1+2/n$, which is also the critical exponent of the semilinear heat equation (see \cite{fujita66}).
 
 However under different class of initial data, the Cauchy problem  \eqref{bdw} admits a different critical exponent. For example, when the initial data satisfies $(u_0,u_1) \in(L^m\cap H^s)\times(L^m\cap L^2)$ with $0<s\le1$ and $(-n+\sqrt{n^2+16n})/4<m<2$, the critical exponent is 
 \begin{align}
 1+2m/n\label{pcm}
 \end{align}
 as demonstrated in  \cite{IkedaInuiOkamotoWkasugi2019,IkehataOhta2002}. 
	
	Recently, the critical exponent of the Cauchy problem \eqref{bdw} with initial data $(u_0,u_1)\in(\Hg\cap H^s)\times(\Hg\cap L^2)$ has been studied in light of the Hardy-Littlewood-Sobolev inequality: $L^m\subset\Hg$ for $1<m\le2$, $0\le \ga<n/2$ and $1/m=1/2+\ga/n$. 
	In the case $1\le n\le6$, Chen and Reissig \cite{chenreissig2023} established that for $(u_0,u_1)\in(\Hg\cap H^s)\times(\Hg\cap L^2)$ with $0\le\ga<\min\{n/2,\tilde{\ga}\}$ here $\tilde{\ga}:=\frac{-n+\sqrt{n^2+16n}}{4}$, the critical exponent for the Cauchy problem \eqref{bdw} is given by
\begin{align}
p_{\mathrm{crit}}:=1+\frac{4}{n+2\ga}\label{pcga}
\end{align}
(See Figure~\ref{imgg}). Then, we remark that $p_{\mathrm{crit}}$ in \eqref{pcga} with $1/m=1/2+\ga/n$, which appears as the condition of Hardy-Littlewood-Sobolev inequality, coinsides with the index in \eqref{pcm}. Moreover,  D'Abbicco \cite{dabbicco2024} proved  the existence of global solutions in the critical case.
 Furthermore, \cite{chenreissig2023} derived sharp lifespan estimates for blow-up weak solutions: if $p<p_{\mathrm{crit}}$, then the lifespan satisfies 
 \begin{align}
 T\les \eps^{-1/(p'-1-n/4-\ga/2)}=\eps^{-1/(p'-p'_\mathrm{crit})}.
 \end{align}
 In particular, if $p>1+2\ga/n$ in addition to $p<p_{\mathrm{crit}}$, then the estimate is sharp 
 \begin{align}
 T\simeq \eps^{-1/(p'-1-n/4-\ga/2)}=\eps^{-1/(p'-p'_\mathrm{crit})}.
\end{align}
	

   		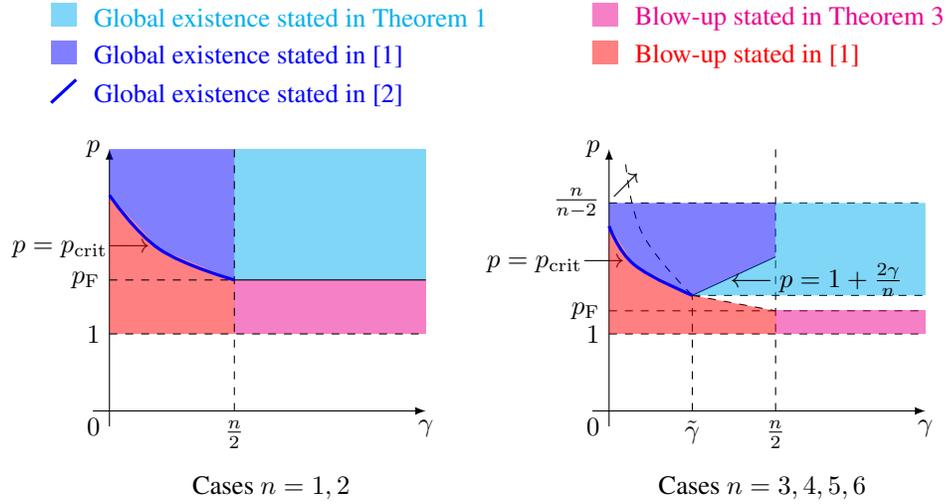
\begin{figure}[http]
			\centering
			\begin{tikzpicture}[>=latex,xscale=1.35,yscale=1.25,scale=0.82]
			\node[right] at (-0.3,4.1) {{\color{blue} Global existence stated in \cite{dabbicco2024}}};
				\node[right] at (-0.3,4.6) {{\color{blue} Global existence stated in \cite{chenreissig2023}}};
				\node[right] at (6.2,4.6) {{\color{red} Blow-up stated in \cite{chenreissig2023}}};
				\node[right] at (-0.3,5.1) {{\color{cyan} Global existence stated in Theorem~\ref{ex}}};
				\node[right] at (6.2,5.1) {{\color{magenta} Blow-up stated in Theorem~\ref{b-u2}}};
				\node[right] at (6.2,4.1) {{\color{brown} Blow-up stated in Theorem~\ref{b-u3}}};
				
				\draw[color=blue,very thick] plot[smooth, tension=.7] coordinates {(-0.7,4)(-0.4,4.3)};
				\fill[opacity = 0.5, color = blue](-0.7,4.5) -- (-0.7, 4.8) -- (-0.4,4.8) -- (-0.4, 4.5);
				\fill[opacity = 0.5, color = cyan](-0.7,5) -- (-0.7, 5.3) -- (-0.4,5.3) -- (-0.4, 5);
				\fill[opacity = 0.5, color = red] (5.8,4.5) -- (5.8,4.8) -- (6.1,4.8) -- (6.1, 4.5);
				\fill[opacity = 0.5, color = magenta] (5.8,5) -- (5.8,5.3) -- (6.1,5.3) -- (6.1, 5);
				\fill[opacity = 0.5, color = brown] (5.8,4) -- (5.8,4.3) -- (6.1,4.3) -- (6.1, 4);
				\draw[->] (-0.2,0) -- (3.8,0) node[below] {$\gamma$};
				\draw[->] (0,-0.2) -- (0,3.4) node[left] {$p$};		
				\node[left, color=black] at (0.6,2.1) {{ $p=p_{\mathrm{crit}}$$\longrightarrow$}};
				\node[left] at (0,-0.2) {{$0$}};
				\fill[opacity = 0.5, color = red]  (0,2.8) -- (0.6, 2.1) -- (1.5,1.7) -- (1.5, 1) -- (0,1);
				\fill[opacity = 0.5, color = magenta]  (1.5,1.7) -- (3.8, 1.7) -- (3.8,1) -- (1.5, 1);
				\fill[opacity = 0.5, color = blue]  (0,2.8) -- (0.6, 2.1) -- (1.5,1.7) -- (1.5,3.4) -- (0, 3.4);
				\fill[opacity = 0.5, color = cyan] (1.5,3.4) -- (3.8,3.4) -- (3.8,1.7) -- (1.5, 1.7);
				\node[below] at (1.5,0) {{$n/2$}};
				\node[left] at (0,1) {{$1$}};
				\node[left] at (0,1.7) {{$p_{\mathrm{F}}$}};
				\draw[dashed, color=black]  (0, 1)--(3.8, 1);
				\draw[dashed, color=black]  (0, 1.7)--(1.5, 1.7);
				\draw[dashed, color=black]  (1.5, 0)--(1.5, 3.4);
				\draw (1.5,1.7) -- (3.8,1.7);
				\draw[->] (5.8,0) -- (9.8,0) node[below] {$\gamma$};
				\draw[->] (6,-0.2) -- (6,3.4) node[left] {$p$};
				\node[left] at (6,-0.2) {{$0$}};
				\node[left] at (6,2.7) {{$\frac{n}{n-2}$}};
				\node[below] at (7,0) {{$\tilde{\ga}$}};
				\node[below] at (8,0) {{$n/2$}};
				\node[left] at (6,1) {{$1$}};
				\node[right] at (9.8,1.5) {{$p=\frac{n+\sqrt{n^2+16n}}{2n}$}};
				
				\node[left, color=black] at (6.3,1.9) {{ $p=p_{\mathrm{crit}}$$\longrightarrow$}};
				\draw[dashed, color=black]  (6, 1)--(9.8, 1);
				\draw[dashed, color=black]  (8, 0)--(8, 3.4);
				\draw[dashed, color=black]  (6, 2.7)--(9.8, 2.7);
				\draw[dashed, color=black]  (7, 0)--(7, 1.5);
				\draw[dashed, color=black] plot[smooth, tension=.7] coordinates {(6.2,3.2) (6.4,2.3) (7,1.5)};
				\node[below] at (6.2,3.2) {{$\nearrow$}};
				\draw[color=black] plot[smooth, tension=.7] coordinates { (7,1.5) (8,2)};
				\draw[dashed, color=black] plot[smooth, tension=.7] coordinates { (7,1.5) (7.5,1.4) (8,1.3)};
				\draw[color=black] (7,1.5) -- (9.8,1.5);
				\fill[opacity = 0.5, color =red]  (6,2.4) -- (6.3,1.9) -- (7,1.5) -- (7.5,1.4) --(8,1.3) -- (8,1) -- (6, 1);
				\fill[opacity = 0.5, color =blue]  (6,2.4) -- (6.3,1.9) -- (7,1.5) -- (8,2) -- (8,2.7) -- (6,2.7);
				\fill[opacity = 0.5, color =cyan] (7,1.5) -- (8,2) -- (8,2.7) -- (9.8,2.7) -- (9.8,1.5);
				\fill[opacity = 0.5, color =brown]  (8,1) -- (8,1.3) -- (7,1.5) -- (9.8,1.5) -- (9.8,1);
				\node[left] at (3,-1) {{Cases $n=1,2$}};
				\node[left] at (9.2,-1) {{Cases $n=3,4,5,6$}};
				\node[left, color=black] at (9.64,1.7) {{$\longleftarrow$ $p=1+\frac{2\gamma}{n}$}};
				\draw[color=blue, very thick] plot[smooth, tension=.7] coordinates {(0,2.8) (0.6,2.1) (1.5,1.7)};
				\draw[color=blue,very thick] plot[smooth, tension=.7] coordinates {(6,2.4) (6.3,1.9) (7,1.5)};
			\end{tikzpicture}
			\caption{Description of the critical exponent in the $\gamma-p$ plane (with $s=1$)}
			\label{imgg}
		\end{figure}
		In this paper, we studied the Cauchy problem \eqref{bdw} when the initial data in $(\Hg\cap H^s)\times(\Hg\cap L^2)$ under conditions that could not be handled in \cite{chenreissig2023}, that is, the case $\ga\ge \min\{n/2,\tilde{\ga}\}$, and we obtain new critical exponents not \eqref{pcga} but 
		\begin{align}
		\begin{cases}
		1+2/n&\text{ if }n=1\text{ or }2,\\
		\frac{n+\sqrt{n^2+16n}}{2n}&\text{ if }3\le n\le 6.
		\end{cases}
		\end{align}
		 Theorem~\ref{ex} in Section \ref{secmainglo} is a result concerning global existence. Theorem~\ref{b-u2} in Section \ref{secmainblow} states that for any $\gamma>0$ solutions blow-up in finite time if $p$ is below the Fujita exponent. These two results is illustrated in Figure~\ref{imgg}. We remark that the critical exponent coincides with the Fujita exponent in particular for $n=1,2$. Theorem~\ref{b-u3} in Section \ref{secmainblow} states that for any $\gamma>0$ solutions blow-up in finite time if $p$ is below $\frac{n+\sqrt{n^2+16n}}{2n}$.

\ \\

\noindent\textbf{Notation:}

We list the notations used throughout this paper. Unless otherwise stated, the following conventions apply.
\begin{itemize}
  \item Constants denote positive real numbers and may change from line to line; $C$ and $C'$ denote a positive constants whose value may vary with context.
  \item The notation $X\les Y$ means there exists $C>0$ such that $X\le C\,Y$. The notation $X\simeq Y$ means $X\les Y$ and $Y\les X$.
  \item For $f\in\mathcal{S}'(\mathbb{R}^n)$, $\mathcal{F}f$ or $\hat f$ denotes its Fourier transform, and $\mathcal{F}^{-1}f$ or $\check f$ denotes the inverse transform.
  \item For $r>0$, set $B(r):=\{x\in\mathbb{R}^n:\,|x|<r\}$. The symbol $\chi_{|\cdot|<r}$ denotes the characteristic function of $B(r)$.
  \item For $p>1$, the conjugate exponent $p'>1$ is defined by $1/p+1/p'=1$.


  \item For $1\le p\le\infty$, the Lebesgue space $L^p(\mathbb{R}^n)$ is defined as usual:

\[
    L^p(\mathbb{R}^n):=\{f \mid \|f\|_{L^p}<\infty\},
  \]

  where for $1\le p<\infty$

\[
    \|f\|_{L^p}:=\Big(\int_{\mathbb{R}^n}|f(x)|^p\,dx\Big)^{1/p},
  \]

  and for $p=\infty$

\[
    \|f\|_{L^\infty}:=\operatorname{ess\,sup}_{x\in\mathbb{R}^n}|f(x)|.
  \]

  \item For $s\in\mathbb{R}$, the (inhomogeneous) Sobolev space $H^s(\mathbb{R}^n)$ is defined by

\[
    H^s(\mathbb{R}^n):=\Big\{f\in\mathcal{S}'(\mathbb{R}^n)\ \Big|\ \|f\|_{H^s}:=\Big(\int_{\mathbb{R}^n}(1+|\xi|^2)^s|\widehat{f}(\xi)|^2\,d\xi\Big)^{1/2}<\infty\Big\},
  \]

  and $H^s(\mathbb{R}^n)$ is a Hilbert space with inner product

\[
    \langle f,g\rangle_{H^s}:=\int_{\mathbb{R}^n}(1+|\xi|^2)^s\widehat{f}(\xi)\overline{\widehat{g}(\xi)}\,d\xi.
  \]

 \item Unless otherwise stated, all function spaces in this paper are taken on the whole space $\mathbb{R}^n$; accordingly $L^p$ and $H^s$ denote $L^p(\mathbb{R}^n)$ and $H^s(\mathbb{R}^n)$ respectively.
\end{itemize}

	\section{Main results of global existence}\label{secmainglo}
	
	This section presents a theorem obtained regarding the existence of global solutions.
	
	We say that a function $u$ is a mild solution of \eqref{bdw} if $u$ satisfies 
	\begin{align}
		u=K'(t)*u_0+K(t)*(u_1+u_0)+\int_0^tK(t-\tau)*|u|^p(\tau) d\tau,
	\end{align}
	where the function $K$ is the fundamental solution of linear dumped wave equation written as
	\begin{align}
		\hat{K}=e^{-\frac{t}{2}}\sum_{k=0}^\infty \frac{\left(\frac{1}{4}-|\xi|^2\right)^k}{(2k+1)!}t^{2k+1}=
		\begin{cases}
			e^{-\frac{t}{2}}\frac{\sinh\left(t\sqrt{\frac{1}{4}-|\xi|^2}\right)}{\sqrt{\frac{1}{4}-|\xi|^2}}&|\xi|<1/2\\
			e^{-\frac{t}{2}}&|\xi|=1/2\\
			e^{-\frac{t}{2}}\frac{\sin\left(t\sqrt{|\xi|^2-\frac{1}{4}}\right)}{\sqrt{|\xi|^2-\frac{1}{4}}}&|\xi|>1/2
		\end{cases}.
	\end{align}
	\begin{thm}\label{ex}
		Let $1\le n\le6$ and let $s\in(0,1]$ suppose that 
		\begin{align}
			\ga\ge \min\left\{n/2, \frac{\sqrt{n^2+16n}-n}{4}\right\}=
			\begin{cases}
				n/2&n=1\text{ or }2\\
				\frac{\sqrt{n^2+16n}-n}{4}&n\ge3
			\end{cases},\label{ga}
		\end{align} 
		\begin{align}
			p>\max\left\{1+\frac{2}{n},\frac{\sqrt{n^2+16n}+n}{2n}\right\}=
			\begin{cases}
				1+2/n & n=1\text{ or }2\\
				\frac{\sqrt{n^2+16n}+n}{2n}& n\ge3
			\end{cases},\label{pga}
		\end{align}
		and $p\le\frac{n}{n-2s}$ if $n>2s$. If
		\begin{align}
			(u_0,u_1)\in(\Hg\cap H^s)\times(\Hg\cap L^2),
		\end{align}then, for a sufficiently small $\eps$, there is a unique mild solution
		\begin{align}
			u\in C([0,\infty):H^{s})
		\end{align}
		to the Cauchy problem for the semilinear damped wave equation \eqref{bdw}. Furthermore, for the number $\tilde{\ga}>0$ satisfing 
		\begin{align}
			\tilde{\ga}<n/2\label{ga'1}
		\end{align} and 
		\begin{align}
			\tilde{\ga}\le\min\left\{\frac{n(p-1)}{2}, \ga\right\},\label{ga'2}
		\end{align}
		the solution $u$ satisfy the following decay estimates:
		\begin{align}
			\|u(t,\cdot)\|_{L^2}&\les
			\eps(1+t)^{-\frac{\tilde{\ga}}{2}}\left(\|(u_0,u_1)\|_{(\Hg\cap H^s)\times(\Hg\cap L^2)}\right)\\
			\|u(t,\cdot)\|_{H^s}&\les
			\eps(1+t)^{-\frac{s+\tilde{\ga}}{2}}\left(\|(u_0,u_1)\|_{(\Hg\cap H^s)\times(\Hg\cap L^2)}\right).
		\end{align}
	\end{thm}
	
		The proof of this theorem utilizes the global existence result established in \cite{chenreissig2023}, together with the fact that, for $0<\ga<\tilde{\ga}$ and $s\ge0$, one has $\dot{H}^{-\tilde{\ga}}\cap H^s\subset \Hg\cap H^s$~. The specific method will be described in Section \ref{secproglo}. 

	\section{Main results of blow-up solution}\label{secmainblow}
	In this section, we establish three blow-up theorems. They are distinguished by the assumptions imposed on $p$ and the initial data, and they provide different lifespan estimates as a result.
	\begin{thm}\label{b-u1}
		Let $n\in\bN, \ga>0$ and $1<p<1+\frac{4}{n+2\ga}$. Assume that the initial data $(u_0,u_1)\in(H^s\cap \Hg)\times(L^2\cap\Hg)$ satisfy $\what{u_0+u_1}\ge0$ and
		\begin{align}
			\what{u_0+u_1}(\xi)\gtr |\xi|^{-n/2+\ga}(\log(|\xi|^{-1}))^{-1}\chi_{|\cdot|<r}(\xi),\label{inHg}
		\end{align}
		for some $r>0$.
		Then, the Cauchy problem \eqref{bdw} has no global (in time) weak solution. 
		Moreover, the lifespan of the local solution $T$ satisfies
		\begin{align}
			T\les\eps^{-\frac{1}{p'-1-\frac{\ga}{2}-\frac{n}{4}-\delta}}\label{lifespan1}
		\end{align} 
		for any $0<\delta\ll1$.
	\end{thm}
	\begin{rem}
		If the initial data satisfy $\what{u_0}=\what{u_1}=|\xi|^{-n/2+\ga}(\log(|\xi|))^{-1}\chi_{|\cdot|<r}(\xi)$, then $(u_0,u_1)\in(\Hg\cap H^s)\times(\Hg\cap L^2)$.
	\end{rem}
	\begin{rem}
		The value $-1/(p'-\ga/2-n/4)$ that appears in the lifespan estimate \eqref{lifespan1} can also be written as$-1/(p'-p'_{crit})$.
	\end{rem}
	\begin{thm}\label{b-u2}
		Let $n\in\bN, 1<p<1+\frac{2}{n}$ and $0<\eps\ll1$. Assume that the initial data $(u_0,u_1)\in(H^s\cap \Hg)\times(L^2\cap\Hg)$ satisfy $\what{u_0+u_1}\ge0$ and 
		\begin{align}
			\what{u_0+u_1}(\xi)>0\ \text{a.e. }\xi\in B(r),
			\end{align}
			for some $r>0$. Here $B(r):=\{\xi\in\bR^n: |\xi|<r\}$. Then, the Cauchy problem \eqref{bdw} has no global (in time) weak solution.
			Moreover, the lifespan of the local solution $T$ satisfies
		\begin{align}
		T\les\eps^{-\frac{p}{p'-1-n/2}}.
		\end{align}
	\end{thm}
	\begin{rem}
		If the initial data satisfy $u_0=u_1=\Delta^k e^{-\frac{|x|^2}{2}}$ with $k\ge\ga/2$, then $(u_0,u_1)\in(\Hg\cap H^s)\times(\Hg\cap L^2)$ and $(u_0,u_1)$ satisfy the conditions of Theorem~\ref{b-u2}.
	\end{rem}
	\begin{rem}
		The value $-p/(p'-n/2)$ that appears in the lifespan estimate \eqref{lifespan1} can also be written as$-p/(p'-p'_\mathrm{F})$.
	\end{rem}
\begin{thm}\label{b-u3}
		Let $n\in\bN, 1<p$ and $0<\eps\ll1$. Assume that the initial data $(u_0,u_1)\in(H^s\cap \Hg)\times(L^2\cap\Hg)$ satisfy $\what{u_0+u_1}\ge0$,
		\begin{align}
		1+\frac{n}{4}p<p'
		\end{align}i.e.
		\begin{align}
		p<\frac{n+\sqrt{n^2+16n}}{2n}
		\end{align}
		and
		\begin{align}
			\what{u_0+u_1}(\xi+a)\gtr|\xi|^{-n/2}(\log(|\xi|^{-1}))^{-1}\ \text{a.e. }\xi\in B(r)\label{inHg2}
			\end{align}
			for some $r>0$ and $0\neq a\in\mathbb{R}^n$. Then, the Cauchy problem \eqref{bdw} has no global (in time) weak solution.
			Moreover, the lifespan of the local solution $T$ satisfies
		\begin{align}
		T\les\eps^{-\frac{p}{p'-1-\frac{n}{4}p-\delta}}\label{b-u3l}
		\end{align}
		for $0<\delta\ll1$.
	\end{thm}
	\begin{rem}
		By Theorem~\ref{b-u1}, Theorem~\ref{b-u2} and Theorem~\ref{b-u3}, we get the upper estimate of lifespan as
		\begin{align}
			T\les \min\left\{\eps^{-\frac{p}{p'-p'_{\mathrm{F}}}}, \eps^{-\frac{1}{p'-p'_{\mathrm{crit}}-\delta}}, \eps^{-\frac{p}{p'-1-\frac{n}{4}p-\delta}} \right\}=
			\begin{cases}
				\eps^{-\frac{p}{p'-p'_{\mathrm{F}}}} &p\ge \max \{2,\frac{4+2n}{n+2\ga}\}\\
				\eps^{-\frac{1}{p'-p'_{\mathrm{crit}}-\delta}} & p\le \min\{\frac{4+2n}{n+2\ga}, \frac{2}{\ga}\}\\
				\eps^{-\frac{p}{p'-1-\frac{n}{4}p-\delta}} & \frac{2}{\ga}<p\le2
			\end{cases}
		\end{align}
		for $1<p<p_{\mathrm{crit}}$, which is illustrated in Figure~\ref{imgg2}.
\begin{figure}
\centering
\begin{tikzpicture}[>=latex,xscale=1.35,yscale=1.25,scale=0.82]

\node[right] at (-2.1,4.2) {{\color{blue} $T\simeq \eps^{-\frac{1}{p'-p'_{\mathrm{crit}}}}$ stated in \cite{chenreissig2023}}};
\node[right] at (2.4,4.2) {{\color{red} $T\lesssim \eps^{-\frac{p}{p'-p'_{\mathrm{F}}}}$ stated in Theorem~\ref{b-u2}}};
\node[right] at (-2.1,4.7) {{\color{cyan} $T\lesssim \eps^{-\frac{1}{p'-p'_{\mathrm{crit}}}}$ stated in \cite{chenreissig2023}}};
\node[right] at (2.4,4.7) {{\color{magenta} $T\lesssim \eps^{-\frac{1}{p'-p'_{\mathrm{crit}}-\delta}}$ stated in Theorem~\ref{b-u1}}};
\node[right] at (2.4,3.7) {{\color{brown} $T\lesssim \eps^{-\frac{p}{p'-1-\frac{n}{4}p-\delta}}$ stated in Theorem~\ref{b-u3}}};

\fill[opacity = 0.5, color =blue]  (-2.5,4) -- (-2.5, 4.3) -- (-2.2,4.3) -- (-2.2, 4);
\fill[opacity = 0.5, color =cyan]  (-2.5,4.5) -- (-2.5, 4.8) -- (-2.2,4.8) -- (-2.2, 4.5);
\fill[opacity = 0.5, color =red]  (2,4) -- (2,4.3) -- (2.3,4.3) -- (2.3, 4);
\fill[opacity = 0.5, color =magenta]  (2,4.5) -- (2,4.8) -- (2.3,4.8) -- (2.3, 4.5);
\fill[opacity = 0.5, color =brown]  (2,3.5) -- (2,3.8) -- (2.3,3.8) -- (2.3, 3.5);

\begin{scope}[shift={(-4,0)}, yscale=0.5]

\draw[->] (-0.2,0) -- (3.7,0) node[below] {$\gamma$};
\draw[->] (0,-0.2) -- (0,6) node[left] {$p$};

\node[left] at (0,-0.2) {$0$};
\node[left] at (0,1) {$1$};

\fill[opacity=0.5, color=blue]
  plot[domain=0:0.5] ({\x},{1+4/(1+2*\x)})
  --
  plot[domain=0.5:0] ({\x},{1+2*\x});

\fill[opacity=0.5, color=cyan]
  plot[domain=0:0.5] ({\x},{1})
  --
  plot[domain=0.5:0] ({\x},{1+2*\x});


\fill[opacity=0.5, color=red]
  plot[domain=0.5:1] ({\x},{(4 + 2)/(1 + 2*\x)})
  --
  plot[domain=1:0.5] ({\x},{3});

\fill[opacity=0.5, color=red]
  plot[domain=1:3.5] ({\x},{2})
  --
  plot[domain=3.5:1] ({\x},{3});


\fill[opacity=0.5, color=magenta]
  plot[domain=0.5:1] ({\x},{1})
  --
  plot[domain=1:0.5] ({\x},{(4 + 2)/(1 + 2*\x)});

\fill[opacity=0.5, color=magenta]
  plot[domain=1:2] ({\x},{1})
  --
  plot[domain=2:1] ({\x},{2/(\x)});


\fill[opacity=0.5, color=brown]
  plot[domain=1:2] ({\x},{2/(\x)})
  --
  plot[domain=2:1] ({\x},{2});

\fill[opacity=0.5, color=brown]
  plot[domain=2:3.5] ({\x},{1})
  --
  plot[domain=3.5:2] ({\x},{2});

\draw[dashed] (0,1) -- (3.7,1);



\draw[dashed, domain=0:0.5]
  plot({\x},{1+4/(1+2*\x)}); 

\draw[domain=0:0.5]
  plot({\x},{1+2*\x});

\draw[dashed, domain=1:2]
  plot({\x},{2/(\x)}); 
\draw[dashed, color=black] (0.5,0) -- (0.5,3);
\node[below] at (0.5,0) {$n/2$};

\draw[dashed, color=black] (0,2) -- (3.7,2);
\node[left] at (0,2) {$2$};

\draw[dashed, color=black] (0,3) -- (3.7,3);

\node[left] at (0,3) {$p_{\mathrm{F}}$};
\draw[dashed, color=black, domain=0.5:1]
  plot({\x},{(4 + 2)/(1 + 2*\x)});

\node[right] at (1,-2) {with $n=1$};

\end{scope}

\begin{scope}[shift={(4,0)}]

\draw[->] (-0.2,0) -- (4.5,0) node[below] {$\gamma$};
\draw[->] (0,-0.2) -- (0,3.4) node[left] {$p$};

\node[left] at (0,-0.2) {{$0$}};
\draw[dashed, color=black] (0,1.7583) -- (4.5,1.7583);
\node[left] at (6,2) {$\displaystyle p=\frac{n+\sqrt{n^2+16n}}{2n}$};

\fill[opacity=0.5, color=blue]
  plot[domain=0:1.1375] ({\x},{1 + 2*\x/3})
  --
  plot[domain=1.1375:0] ({\x},{1 + 4/(3 + 2*\x)});


\fill[opacity=0.5, color=cyan]
  plot[domain=0:1.137458] ({\x},{1})
  --
  plot[domain=1.137458:0] ({\x},{1 + 2*\x/3});

\fill[opacity=0.5, color=cyan]
  plot[domain=1.137458:1.5] ({\x},{1})
  --
  plot[domain=1.5:1.137458] ({\x},{2/(\x)});

\fill[opacity=0.5, color=magenta]
  plot[domain=1.5:2] ({\x},{1})
  --
  plot[domain=2:1.5] ({\x},{2/(\x)});


\fill[opacity=0.5, color=brown]
  plot[domain=1.137458:2] ({\x},{2/(\x)})
  --
  plot[domain=2:1.137458] ({\x},{1.7583});

\fill[opacity=0.5, color=brown]
  plot[domain=2:4.5] ({\x},{1})
  --
  plot[domain=4.5:2] ({\x},{1.7583});

\draw[dashed, color=black, domain=1.1:2]
  plot({\x},{2/(\x)}) node[below] {$p=\frac{2}{\gamma}$};

\draw[dashed, color=black]  (0, 1)--(4.5, 1);
\draw[dashed, color=black]  (1.5, 0)--(1.5, 3.4);

\draw (1.5,1)--(1.5,1.7);
\draw[dashed] (0,1)--(1.1,1.75);
\draw[dashed, domain=0:4.5] 
  plot({\x},{1+4/((2*\x)+3)}) node[right]{$p=p_{\mathrm{crit}}$};

\node[right] at (1,-1) {with $2\le n\le6$};

\end{scope}

\end{tikzpicture}
\caption{Description of the lifespan in the $\gamma-p$ plane}
\label{imgg2}
\end{figure}
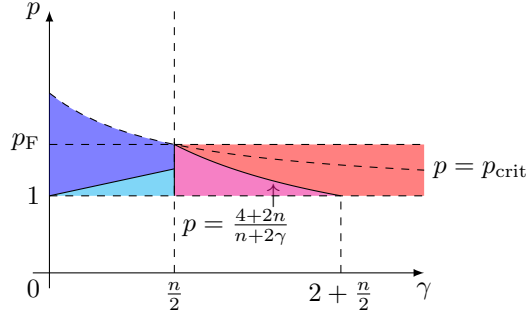

	\end{rem}
	\section{Proof of global existence}\label{secproglo}
	In this section, we give the proof of Theorem~\ref{ex}. Before proceeding, we state the following result of \cite{chenreissig2023}. 
	\begin{lem}[{\protect \cite[Theorem1]{chenreissig2023}}]\label{crt1}
		Let $n\le6$, $s\in(0,1]$ and $0<\ga<n/2$. Let $p$ fulfil $p>1+\frac{4}{2\ga+n}$, $p\ge 1+\frac{2\ga}{n}$ and $p\le\frac{n}{n-2s}$ if $n>2s$. Let us assume 
		\begin{align}
			(u_0,u_1)\in(\Hg\cap H^s)\times(\Hg\cap L^2).
		\end{align}Then, for a sufficiently small $\eps$, there is a unique mild solution
		\begin{align}
			u\in C([0,\infty):H^{s})
		\end{align}
		to the Cauchy problem for the semilinear damped wave equation \eqref{bdw}. Furthermore, 
		the solution $u$ satisfies the following decay estimates:
		\begin{align}
			\|u(t,\cdot)\|_{L^2}&\les
			\eps(1+t)^{-\frac{\ga}{2}}\left(\|(u_0,u_1)\|_{(\Hg\cap H^s)\times(\Hg\cap L^2)}\right)\\
			\|u(t,\cdot)\|_{H^s}&\les
			\eps(1+t)^{-\frac{s+\ga}{2}}\left(\|(u_0,u_1)\|_{(\Hg\cap H^s)\times(\Hg\cap L^2)}\right).
		\end{align}
	\end{lem}
	 Before proving Theorem~\ref{ex}, we state that $(\Hg\cap H^s)\times(\Hg\cap L^2)$ is monotonic with the parametar $\ga$.
	 \begin{lem}\label{monotonic}
	 Let $s\ge0$ be fixed. Then for any parameters satisfying $0\le\tilde{\ga}<{\ga}$, the following inclusion holds:
	 \begin{align}
	 H^s\cap\dot{H}^{-{\ga}}\subset H^s\cap\dot{H}^{-\tilde{\ga}}.
	 \end{align}
	 \end{lem}
	 \begin{proof}[\textbf{Proof of Lemma~{\ref{monotonic}}}]
	
		\begin{align}
			&\|f\|_{H^{s}\cap \dot{H}^{-\tilde{\ga}}}\\
			\simeq&\|\langle\xi\rangle^{s}\hat{f}\|_{L^2}+\||\xi|^{-\tilde{\ga}}\hat{f}\|_{L^2}\\
			\les&\|\chi_{|\xi|\ge1}(\xi)\langle\xi\rangle^{s}\hat{f}\|_{L^2}+\|\chi_{|\cdot|\le1}(\xi)|\xi|^{-\tilde{\ga}}\|_{L^2}\\
			\le&\|\chi_{|\cdot|\ge1}(\xi)\langle\xi\rangle^{s}\hat{f}\|_{L^2}+\|\chi_{|\cdot|\le1}(\xi)|\xi|^{-\ga}\|_{L^2}\\
			\les&\|f\|_{H^{s}\cap \dot{H}^{-\ga}}.
		\end{align}
	 \end{proof}
	\begin{proof}[\textbf{Proof of Theorem~{\ref{ex}}}]
		We may assume without loss of generality that 
		\begin{align}
			2p'-2-n/2<\tilde{\ga}.
		\end{align}
		
		By \eqref{ga'2} and Lemma \ref{monotonic}, we get $(\Hg\cap H^s)\times(\Hg\cap L^2)\subset(H^s\cap\dot{H}^{-\tilde{\ga}})\times(L^2\cap\dot{H}^{-\tilde{\ga}})$.
		By \eqref{pga}, \eqref{ga'1} and \eqref{ga'2}, $\tilde{\ga}$ satisfies the conditions of the parameter $\ga$ in Lemma  \ref{crt1}. Hence, we invoke Lemma \ref{crt1} with $\ga$ repleced $\ga'$ in \eqref{ga'1} and \eqref{ga'2} to get Theorem~\ref{ex}.
	\end{proof}
	\section{Proof of blow-up theorems}\label{secproblow}
	In this section, we give the proof of Theorem~\ref{b-u1} and Theorem~\ref{b-u2}.
	\subsection{Construction of a bump function used in the proof}
We need a non-trivial bump function $\phi\in C_0^{\infty}(\bR^n)$ which satisfies the following conditions. For the proof, we need some bump function given by the following lemma.

	\begin{lem}
	There exists a non-trivial function $\phi\in C_0^{\infty}(\bR^n)$ which satisfies the following conditions:
		\begin{enumerate}[label=\textup{(\roman*)}]
	\item\label{i} $\phi\ge0$ .
	\item\label{ii}$\what{\phi}\ge0$ .
	\item\label{iii}$\phi(Rx)\le\phi(rx)$ for all $0<r<R\text{ and } x\in\bR^n$. 
	\end{enumerate}
	\end{lem}
	\begin{proof}
	Let $\tilde{\phi}\in C_0^\infty(\bR^n)$ satisfy \tref{i}, \tref{iii} and
	\begin{align}
	|x|=|y|\Rightarrow\tilde{\phi}(x)=\tilde{\phi}(y)\ \ \ \text{for all } x,y\in\bR^n.\label{newcon}
	\end{align}Let $\eta\in C_0^{\infty}\left([0,\infty)\right)$ satisfy $\tilde{\phi}(x)=\eta(|x|)$. Then, $\eta'\le0$.
	 For example, function $f$ such that
	\begin{align}
	f(x)=
	\begin{cases}
	e^{-\frac{1}{1-|x|^2}}&|x|<1\\
	0&|x|\ge1
	\end{cases}.
	\end{align}Then, $\tilde{\phi}$ is a real-valued function.
	Now define the bump function $\phi\coloneqq\tilde{\phi}*\tilde{\phi}$. We will show that $\phi$ satisfies the conditions \tref{i}, \tref{ii} and \tref{iii}.  
	\begin{enumerate}[label=(\tr{\roman*})]
	\item Because of $\tilde{\phi}\ge0$, we have $\phi=\tilde{\phi}*\tilde{\phi}\ge0$.
	\item$\what{\phi}$ satisfies that $\what{\phi}=\what{\tilde{\phi}*\tilde{\phi}}=\what{\tilde{\phi}}^2$. Where, $\tilde{\phi}$ is real-valued and an even function because of \eqref{newcon}. Therefore, $\what{\tilde{\phi}}$ is real-valued and $\what{\phi}$ is non-negative function.
	\item To prove \tr{(iii)}, it is sufficient to prove that $\pr_R(\phi(Rx))\le0$. 
	\begin{align}
	&\pr_R(\phi(Rx))\\
	=&x\cdot(\nabla\phi)(Rx)\\
	=&\int_{\bR^n}x\cdot(\nabla\tilde{\phi})(y)\tilde{\phi}(Rx-y)dy\\
	=&\int_{\bR^n}x\cdot\left(\frac{y}{|y|}\eta'(|y|)\right)\tilde{\phi}(Rx-y)dy\\
	=&\int_{x\cdot y\le0}x\cdot\left(\frac{y}{|y|}\eta'(|y|)\right)\tilde{\phi}(Rx-y)dy
	+\int_{x\cdot y\ge0}x\cdot\left(\frac{y}{|y|}\eta'(|y|)\right)\tilde{\phi}(Rx-y)dy.\label{half}
	\end{align}
	To change the variable to $y'=-y$, the first half of \eqref{half} is
	\begin{align}
	\int_{x\cdot y'\ge0}x\cdot\left(\frac{-y'}{|y'|}\eta'(|y'|)\right)\tilde{\phi}(Rx+y')dy'.
	\end{align} So, with attention for $|y|=|-y|$, \eqref{half} is
	\begin{align}
	&-\int_{x\cdot y\ge0}x\cdot\left(\frac{y}{|y|}\eta'(|y|)\right)\tilde{\phi}(Rx+y)dy
	+\int_{x\cdot y\ge0}x\cdot\left(\frac{y}{|y|}\eta'(|y|)\right)\tilde{\phi}(Rx-y)dy\\
	=&\int_{x\cdot y\ge0}x\cdot\left(\frac{y}{|y|}\eta'(|y|)\right)(\tilde{\phi}(Rx+y)-\tilde{\phi}(Rx+y))dy\\
	=&\int_{x\cdot y\ge0}x\cdot\left(\frac{y}{|y|}\eta'(|y|)\right)(\eta(|Rx-y|)-\eta(|Rx+y|))dy\\
	=&\int_{x\cdot y\ge0}\left(x\cdot\frac{y}{|y|}\right)\eta'(|y|)(\eta(|Rx-y|)-\eta(|Rx+y|))dy.\label{le0}
	\end{align}
	On the integrand function in \eqref{le0}, since $x\cdot y\ge0$. we have $|Rx-y|\le|Rx+y|$. Thus, $\eta(|Rx-y|)-\eta(|Rx+y|)\ge0$. Moreover, we have $x\cdot(y/|y|)\ge0$ and $\eta'(|y|)\le0$. So, we obtain $\eqref{le0}\le0$.
	\end{enumerate}
	\end{proof}
	\begin{lem}
	Let $\phi\in C_0^{\infty}(\bR^n)$ satisfy the conditions \tref{i}, \tref{ii} and \tref{iii}. Then, for all $k\in\bN$, $\phi^k$ also satisfies those conditions.
	\end{lem}
	\begin{proof}
	We will show that $\phi^k$ satisfies \tref{i}, \tref{ii} and \tref{iii}. Then, since 
	\begin{enumerate}[label=(\tr{\roman*})]
	\item Since $\phi\ge0$, we get $\phi^k\ge0$.
	\item We have 
	\begin{align}
	\what{\phi^k}=\what{\phi}*\ldots*\what{\phi}.
	\end{align}Then, since $\what{\phi}\ge0$, We get $\what{\phi^k}\ge0$. 
	\item Since $\phi$ satisfies \tref{iii}, we get that $\phi^k$ satisfies \tref{iii}.
	\end{enumerate}
	\end{proof}
	\subsection{Proof of the blow-up theorems}
	Let $u$ be the local weak solution of \eqref{bdw}. If the lifespan $T$ is in $[0,1]$, then we get $T\le\eps^{-a}\text{ for all }a>0\text{ and }0<\eps\ll1$. Therefore, let us assume $T>1$. Let $\tilde{\eta}\in C_0^\infty{[0,\infty)}$ be a monotonically non-increasing function that satisfies $\supp\tilde{\eta}\subset[0,1]$ and $\tilde{\eta}(t)=1 \text{ for all }t\in[0,1/2]$. Let $\tilde{\phi}\in C_0^{\infty}(\bR^n)$ satisfy \tref{i}, \tref{ii} and \tref{iii}. We take $l\in\bN$ being $l>2p'$ and we set $\phi\coloneqq\tilde\phi^l$ and $\eta\coloneqq\tilde\eta^l$. We set ${\phi}_R(x)\coq\tilde{\phi}(R^{-1}x)\text{ and }\eta_R(t)\coq\eta(R^{-2}t)$ for all $R>0$. Let a function $I:[1,\sqrt{T})\to\bR$ be
	\begin{align}
	I(R)\coq\iint_{\bR^n\times[0,T)}|u(x,t)|^p\phi_R(x)\eta_R(t)dxdt.\label{I}
	\end{align}
	
	First, we will show a lemma for the theorems of blow-up.
	\begin{lem}\label{1lem}
	The function $I(R)$ satisfies
	\begin{align}
	I(R)\le-\eps\int_{\bR^n}(u_0+u_1)\phi_R(x)dx+CI(R)^{1/p}R^{\frac{n+2}{p'}-2}\label{lem}
	\end{align}
	and
	\begin{align}
	I(R)\les-\eps\int_{\bR^n}(u_0+u_1)\phi_R(x)dx+CR^{n+2-2p'}\label{lem2}
	\end{align}
	for all $R\in[1,\sqrt{T})$ with a constantant $C>0$ .
	\end{lem}
	\begin{proof}[\textbf{Proof of Lemma \ref{1lem}}]
	First, since $u$ is weak solution of the Cauchy problem \eqref{bdw}, we have
	\begin{align}
	&I(R)\\
	=&-\eps\int_{\bR^n}(u_0+u_1)\phi_R(x)dx\\
	&+\iint_{\bR^n\times[0,T)}u\left(\prtt-\prt-\Delta\right)(\phi_R\eta_R)dxdt\label{27}.
	\end{align}
	We further estimate the term \eqref{27} as
	\begin{align}
	&\iint_{\bR^n\times[0,T)}u\left(\prtt-\prt-\Delta\right)(\phi_R\eta_R)dxdt\\
	=&\iint_{\supp\phi_R\eta_R}\left(|u|^p\phi_R\eta_R\right)^{1/p}\phi_R^{-1/p}\eta_{R}^{-1/p}\left(\prtt-\prt-\Delta\right)(\phi_R\eta_R)dxdt\\
	\le&I(R)^{1/p}\left(\iint_{\supp\phi_R\eta_R}\phi_R^{-p'/p}\eta_R^{-p'/p}\left|\left(\prtt-\prt-\Delta\right)(\phi_R\eta_R)\right|^{p'}dxdt\right)^{1/p'}\ \ (\text{by H\"older's inequality}).\\
	=&I(R)^{1/p}\left(\iint_{\supp\phi_R\eta_R}\phi^{-p'/p}\eta^{-p'/p}\left|\left(R^{-4}\frac{\pr^2}{\pr\tau^2}-R^{-2}\frac{\pr}{\pr\tau}-R^{-2}\Delta\right)(\phi\eta)\right|^{p'}R^{n+2}dyd\tau\right)^{1/p'}\\
	\les&I(R)^{1/p}R^{\frac{n+2}{p'}-2}\left(\iint_{\supp\phi\eta}\phi^{-p'/p}\eta^{-p'/p}\left|\left(\frac{\pr^2}{\pr\tau^2}-\frac{\pr}{\pr\tau}-\Delta\right)(\phi\eta)\right|^{p'}dyd\tau\right)^{1/p'}\label{Integral Check}
	\end{align}Here, we will show the value
	\begin{align}
	\iint_{\supp\phi\eta}\phi^{-p'/p}\eta^{-p'/p}\left|\left(\frac{\pr^2}{\pr\tau^2}-\frac{\pr}{\pr\tau}-\Delta\right)(\phi\eta)\right|^{p'}dyd\tau
	\end{align}is finite. The integrand function is
	\begin{align}
	&\phi^{-p'/p}\eta^{-p'/p}\left|\left(\frac{\pr^2}{\pr\tau^2}-\frac{\pr}{\pr\tau}-\Delta\right)(\phi\eta)\right|^{p'}\\
	=&\left|\phi^{-1/p}\eta^{-1/p}\left(\frac{\pr^2}{\pr\tau^2}-\frac{\pr}{\pr\tau}-\Delta\right)(\phi\eta)\right|^{p'}\label{24}
	\end{align}
	and, since the integral domain of \eqref{Integral Check} is bounded, it is sufficient to show that \eqref{24} is bounded. To do this, consider $\phi=\tilde{\phi}^l,\ \eta=\tilde{\eta}^l$, and express \eqref{24} in terms of $\tilde{\phi},\tilde{\eta}$ instead of $\phi,\eta$. We will express $\eta',\eta'',\Delta\phi$ as $\tilde{\eta},\tilde{\phi}$.
	\begin{align}
	&\eta'=(\tilde{\eta}^l)'=l\tilde{\eta}'\tilde{\eta}^{l-1}\\
	&\eta''=(\tilde{\eta}^l)''=l(l-1)(\tilde{\eta}')^{2}\tilde{\eta}^{l-2}+l\tilde{\eta}''\tilde{\eta}^{l-1}\\
	&\Delta\phi=\Delta(\tilde{\phi}^l)=l(l-1)|\nabla\tilde{\phi}|^2\tilde{\phi}^{l-2}+l(\Delta\tilde{\phi})\tilde{\phi}^{l-1}.
	\end{align}And, we get
	\begin{align}
	&\phi^{-1/p}\eta^{-1/p}\frac{\pr^2}{\pr\tau^2}(\phi\eta)\\
	=&\tilde{\phi}^{-l/p}\tilde{\eta}^{-l/p}\tilde{\phi}^l(l(l-1)(\tilde{\eta}')^{2}\tilde{\eta}^{l-2}+l\tilde{\eta}''\tilde{\eta}^{l-1})\\
	=&\tilde{\phi}^{l(1-1/p)}(l(l-1)(\tilde{\eta})')^2\tilde{\eta}^{l(1-1/p)-2}+l\tilde{\eta}''\tilde{\eta}^{l(1-1/p)-1})\\
	=&\tilde{\phi}^{l/p'}(l(l-1)(\tilde{\eta})')^2\tilde{\eta}^{l/p'-2}+l\tilde{\eta}''\tilde{\eta}^{l/p'-1}),\label{29}\\
	&\phi^{-1/p}\eta^{-1/p}\frac{\pr}{\pr\tau}(\phi\eta)\\
	=&\tilde{\phi}^{-l/p}\tilde{\eta}^{-1/p}\tilde{\phi}^{l}(l\tilde{\eta}'\tilde{\eta}^{l-1})\\
	=&l\tilde{\phi}^{l/p'}\tilde{\eta}^{l/p'-1}\tilde{\eta}'\label{32},\\
	&\phi^{-1/p}\eta^{-1/p}\Delta(\phi\eta)\\
	=&\tilde{\phi}^{-l/p}\tilde{\eta}^{-1/p}(l(l-1)|\nabla\tilde{\phi}|^2\tilde{\phi}^{l-2}+l(\Delta\tilde{\phi})\tilde{\phi}^{l-1})\tilde{\eta}^{l}\\
	=&\tilde{\eta}^{l(1-1/p)}(l(l-1)|\nabla\tilde{\phi}|^2\tilde{\phi}^{l(1-1/p)-2}+l(\Delta\tilde{\phi})\tilde{\phi}^{l(1-1/p)-1})\\
	=&\tilde{\eta}^{l/p'}(l(l-1)|\nabla\tilde{\phi}|^2\tilde{\phi}^{l/p'-2}+l(\Delta\tilde{\phi})\tilde{\phi}^{l/p'-1})\label{36}
	\end{align}
	$\eqref{24}=|\eqref{29}+\eqref{32}+\eqref{36}|^{p'}$ and, the values appearing in the powers of \eqref{29}, \eqref{32}, and \eqref{36} are all positive (note that $l>2p'$). So \eqref{24} is bounded, and we get \eqref{lem}.
	
	Next, we will show \eqref{lem2} to use \eqref{lem}. Applying Young's inequality, the second term in \eqref{lem} is bounded by $CR^{n+2-p'}+I(R)/p$.
	\end{proof}
	\begin{proof}[\textbf{Proof of Theorem~\ref{b-u1}}]
	First, we will show that there exists a function in $(\Hg\cap H^s)\times(\Hg\cap L^2)$ which satisfies the condition \eqref{inHg}. It is sufficient to show the next lemma.
	\begin{lem}\label{inHglem}With the conditions of Theorem~\ref{b-u1}, we have
	\begin{align}
	\mathcal{F}^{-1}\left(|\xi|^{-n/2+\ga}(\log(|\xi|^{-1}))^{-1}\chi_{|\cdot|<r}(\xi)\right)\in L^2\times\Hg.
	\end{align}
	\end{lem}
	\begin{proof}[\textbf{Proof of Lemma~\ref{inHglem}}]
	 It is sufficient to show that
	\begin{align}
	f(\xi):=|\xi|^{-n/2+\ga}(\log(|\xi|^{-1}))^{-1}\chi_{|\cdot|<r}(\xi)\in L^2
	\end{align}
	and
	\begin{align}
	g(\xi):=|\xi|^{-n/2}(\log(|\xi|^{-1}))^{-1}\chi_{|\cdot|<r}(\xi)\in L^2.\label{f(xi)}
	\end{align}
	To show them, it is sufficient to show only $g\in L^2$. We have
	\begin{align}
	&\int_{\bR^n}|g(\xi)|^2d\xi\\
	=&\int_{\bR^n}|\xi|^{-n}(\log(|\xi|))^{-2}\chi_{|\cdot|<1}(\xi)d\xi\\
	=&\int_{|\xi|<r}|\xi|^{-n}(\log(|\xi|))^{-2}d\xi.\label{xi=romega}
	\end{align}
	by change of variable $\xi=e^{l}\omega(-\infty<l<0,\omega/r\in S^{n-1})$, we get
	\begin{align}
	\eqref{xi=romega}&=\iint_{l<0,\omega/r\in S^{n-1}}e^{-nl}l^{-2}e^{nl}dld\omega\\
	&\simeq\int_{-\infty}^0l^{-2}dl<\infty.
	\end{align}
	\end{proof}
	With the conditions of initial data,
	\begin{align}
	&\int_{\bR^n}(u_0+u_1)\phi_R(x)dx\\
	=&\int_{\bR^n}\what{u_0+u_1}\check{\phi_R}(\xi)d\xi\\
	=&\int_{\bR^n}\what{u_0+u_1}(\xi)R^{n}\check{\phi}(R\xi)d\xi\\
	=&\int_{\bR^n}\what{u_0+u_1}(R^{-1}\xi)\check{\phi}(\xi)d\xi\\
	\gtr&\int_{|\xi|\le rR}R^{n/2-\ga}|\xi|^{-n/2+\ga}(\log({R|\xi|^{-1}}))^{-1}\check{\phi}(\xi)d\xi.\label{ini2}
	\end{align}Here, the term $(\log({R|\xi|^{-1}}))^{-1}$ satisfies (note $|\xi|\le rR$)
	\begin{align}
	(\log({R|\xi|^{-1}}))^{-1}&=\left(\log R+\log(|\xi|^{-1})\right)^{-1}\\
	&=\left(\log R\left(1+\frac{\log(|\xi|^{-1})}{\log R}\right)\right)^{-1}\\
	&\ge \left(\log R\left(1+\frac{\log(|\xi|^{-1})}{\log R}\right)\right)^{-1}\\
	&\ge \left(\log R\left(1+\frac{\log(|\xi|^{-1})}{\log( r^{-1}|\xi|)}\right)\right)^{-1}\\
	&= \left(\log R\right)^{-1}\left(\frac{\log( r^{-1}|\xi|)+\log(|\xi|^{-1})}{\log( r^{-1}|\xi|)}\right)^{-1}\\
	&= \left(\log R\right)^{-1}\left(\frac{\log( r^{-1})}{\log( r^{-1}|\xi|)}\right)^{-1}\\
	&= \left(\log R\right)^{-1}\frac{\log( r^{-1}|\xi|)}{\log( r^{-1})}.
	\end{align}
	Therefore, we get 
	\begin{align}
	&\int_{\bR^n}(u_0+u_1)\phi_R(x)dx\gtr R^{n/2-\ga}(\log R)^{-1}
	\end{align}and
	\begin{align}
	I(R)\les-\eps R^{n/2-\ga}(\log R)^{-1}+CR^{n+2-2p'}.
	\end{align}Since $I(R)\ge0$, we get $-\eps R^{n/2-\ga}(\log R)^{-1}+CR^{n+2-2p'}\ge0$ and
	\begin{align}
	\eps\les R^{n/2+2-2p'+\ga}(\log R)\le R^{n/2+\ga+2-2p'+\delta}\stackrel{R\to \sqrt{T}}{\longrightarrow}T^{\frac{n}{4}+\frac{\ga}{2}+1-p'+\delta}.
	\end{align}So, we get $T\les \eps^{-\frac{1}{p'-(1+n/4+\ga/2)}}$.
	\end{proof}
	\begin{proof}[\textbf{Proof of Theorem~\ref{b-u2}}]
	First we note
	\begin{align}
	&\int_{\bR^n}(u_0+u_1)\phi_Rdx\\
	=&\int_{\bR^n}\what{u_0+u_1}\check{\phi_R}(\xi)d\xi.\label{ini>0}
	\end{align}
	Since $\check{\phi_R}(0)>0$ (because of \ref{iii} and the fact that $\phi_R$ is a positive non-trivial function with compact support ) and $\hat{u_0}, \hat{u_1}>0$, we get $\eqref{ini>0}>0$. For \eqref{lem} and $I(R)>0$, we get
	\begin{align}
	0&\les -\eps\int_{\bR^n}(u_0+u_1)\phi_Rdx+CI(R)^{1/p}R^{\frac{n+2}{p'}-2}
	\end{align}
	and
	\begin{align}
	I(R)\ge\eps^{p}\left(\int_{\bR^n}(u_0+u_1)\phi_RdxR^{\frac{n+2}{p'}-2}\right)^p(\eqqcolon \eps^pg(R)).
	\end{align}
	Because of the definition of $I(R)$, $I$ is a monotonically non-decreasing function. Therefore, we get
	\begin{align}
	\eps^pg(1)\les I(1)\le I(R)&\les -\eps\int_{\bR^n}(u_0+u_1)\phi_Rdx+CR^{n+2-2p'}\\
	&\les R^{n+2-2p'}\stackrel{R\to \sqrt{T}}{\longrightarrow} T^{n/2+1-p'}.
	\end{align}So, we get $T\les \eps^{-\frac{p}{p'-(1+n/2)}}$.
	\end{proof}
	\begin{proof}[\textbf{Proof of Theorem~\ref{b-u3}}]
	First, we will show that there exists a function in $(\Hg\cap H^s)\times(\Hg\cap L^2)$ which satisfies the condition \eqref{inHg2}. It is sufficient to show the next lemma.
	\begin{lem}\label{inHglem2}With the conditions of Theorem~\ref{b-u3}, we have
	\begin{align}
	\mathcal{F}^{-1}\left(|\xi-a|^{-n/2}(\log(|\xi-a|^{-1}))^{-1}\chi_{|\cdot|<1}(\xi-a)\right)\in L^2\times\Hg.
	\end{align}
	\end{lem}
	\begin{proof}[\textbf{Proof of Lemma \ref{inHglem2}}]
	 It is sufficient to show that
	\begin{align}
	\tilde{f}(\xi):=|\xi-a|^{-n/2}(\log(|\xi-a|^{-1}))^{-1}\chi_{|\cdot|<r}(\xi-a)\in L^2\label{f(xi)2}
	\end{align}
	and
	\begin{align}
	\tilde{g}(\xi):=|\xi|^{-\ga}|\xi-a|^{-n/2}(\log(|\xi-a|^{-1}))^{-1}\chi_{|\cdot|<r}(\xi-a)\in L^2.\label{g(xi)2}
	\end{align}
	\eqref{f(xi)2} is obvious form \eqref{f(xi)} by considering $\tilde{f}(\xi+a)=g(\xi)$. Regarding \eqref{g(xi)2}, Since $0\notin\supp\tilde{g}$ (Take $r$ in such a way that it is sufficiently small.), then $|\xi^{-\ga}|$  is bounded. Therefore $\tilde{g}\in L^2$.
	\end{proof}
	First, we define the function $J$ following
	\begin{align}
	J(R)\coq\iint_{\bR^n\times[0,T)}|u(x,t)|^p(2+\cos(a\cdot x))\phi_R(x)\eta_{1/2}(t)dxdt.\label{J}
	\end{align} 
	The proof of this theorem uses the following lemma, which expresses the relationship between $J$ and $I$.
	\begin{lem}\label{lemin}
The functions $I(R),\ J(R)$ satisfy
	\begin{align}
	J(R)\les -\eps R^{n/2}(\log R)^{-1}+CI(R)^{1/p}R^{n/p'}\label{lem3}
	\end{align}
	for all $R\in[1,\sqrt{T})$ with a constantant $C>0$ .
\end{lem}
\begin{proof}[\textbf{Proof of Lemma \ref{lemin}}]
First, since $u$ is weak solution of the Cauchy problem \eqref{bdw}, we have
	\begin{align}
	&J(R)\\
	=&-\eps\int_{\bR^n}(u_0+u_1)(2+\cos(a\cdot x))\phi_R(x)dx\label{term1}\\
	&+\iint_{\bR^n\times[0,1)}u\left(\prtt-\prt-\Delta\right)((2+\cos(a\cdot x))\phi_R\eta_{1/2})dxdt\label{27}.
	\end{align}
	We further estimate the term \eqref{27} as
	\begin{align}
	&\iint_{\bR^n\times[0,T)}u\left(\prtt-\prt-\Delta\right)((2+\cos(a\cdot x))\phi_R\eta_{1/2})dxdt\\
	=&\iint_{\supp\phi_R\eta_R}\left(|u|^p\phi_R\eta_R\right)^{1/p}\phi_R^{-1/p}\eta_{R}^{-1/p}\left(\prtt-\prt-\Delta\right)((2+\cos(a\cdot x))\phi_R\eta_{1/2})dxdt\\
	\le&I(R)^{1/p}\left(\iint_{\supp\phi_R\eta_{1/2}}\phi_R^{-p'/p}\eta_R^{-p'/p}\left|\left(\prtt-\prt-\Delta\right)((2+\cos(a\cdot x))\phi_R\eta_{1/2})\right|^{p'}dxdt\right)^{1/p'}.\label{etain}
	\end{align}
Considering the values of the functions $\eta_{1/2}$ and $\eta_R$, we can see that $\eta_R=1$ with $t\in\supp\eta_{1/2}$. So, \eqref{etain} is
	\begin{align}
	&I(R)^{1/p}\left(\iint_{\supp\phi_R\eta_{1/2}}\phi_R^{-p'/p}\left|\left(\prtt-\prt-\Delta\right)((2+\cos(a\cdot x))\phi_R\eta_{1/2})\right|^{p'}dxdt\right)^{1/p'}.\label{etain2}
\end{align}
This leads us to examine $(2+\cos(a\cdot x))\phi_R$. Calculation yields
\begin{align}
&|\Delta((2+\cos(a\cdot x))\phi_R(x))|\\
=&|\Delta(2+\cos(a\cdot x))\phi_R(x)+2\nabla(2+\cos(a\cdot x))\cdot\nabla\phi_R(x)+(2+\cos(a\cdot x))\Delta\phi_R|\\
\le&|a|^2|\cos(a\cdot x)|\phi_R(x)+R^{-1}|\sin(a\cdot x)a||(\nabla\phi)(R^{-1}x)|+R^{-2}(2+\cos(a\cdot x))(\Delta\phi)(R^{-1}x)\\
\les&\phi(R^{-1}x)+|(\nabla\phi)(R^{-1}x)|+|(\Delta\phi)(R^{-1}x)|.
\end{align}
There, we define $g:=\phi+|\nabla\phi|+|\Delta\phi|$.  And, we can get $(2+\cos(a\cdot x))\phi_R\les \phi_R(x)$. Substituting these results into the integral in \eqref{etain2}, we obtain
\begin{align}
&\iint_{\supp\phi_R\eta_{1/2}}\phi_R^{-p'/p}\left|\left(\prtt-\prt-\Delta\right)((2+\cos(a\cdot x))\phi_R\eta_{1/2})\right|^{p'}dxdt\\
\les&\iint_{\supp\phi_R\eta_{1/2}}\phi_R|\eta''_{1/2}-\eta'_{1/2}|^{p'}dxdt\\
&+\iint_{\supp\phi_R\eta_{1/2}}\eta_{1/2}^{p'}\left(\phi_R+\phi^{-1/p}|(\nabla\phi)(R^{-1}x)|+\phi^{-1/p}|(\Delta\phi)(R^{-1}x)|\right)^{p'}dxdt\\
\les&R^{n}\iint_{\supp\phi\eta_{1/2}}\phi|\eta''_{1/2}-\eta'_{1/2}|^{p'}dxdt\\
&+R^n\iint_{\supp\phi\eta_{1/2}}\eta_{1/2}^{p'}\left(\phi+\phi^{-1/p}|\nabla\phi|+\phi^{-1/p}|\Delta\phi|\right)^{p'}dxdt.\label{Integral Check}
\end{align}

	Here we verify that the integral of \eqref{Integral Check} is finite. Since this integral domain is bounded, it is sufficient to show that \eqref{Integral Check} is bounded. To do this, consider $\phi=\tilde{\phi}^l,\ \eta=\tilde{\eta}^l$, and express \eqref{Integral Check} in terms of $\tilde{\phi}$ instead of $\phi$. We will express $\eta',\eta'',\Delta\phi$ as $\tilde{\eta},\tilde{\phi}$.
	\begin{align}
	&\nabla\phi=\nabla(\tilde{\phi})=l\tilde{\phi}^{l-1}\nabla\tilde{\phi}\\
	&\Delta\phi=\Delta(\tilde{\phi}^l)=l(l-1)|\nabla\tilde{\phi}|^2\tilde{\phi}^{l-2}+l(\Delta\tilde{\phi})\tilde{\phi}^{l-1}.
	\end{align}And, we get
	\begin{align}
	\phi^{-1/p}|\nabla\phi|&=l\tilde{\phi}^{-l/p+l-1}|\nabla\phi|=l\tilde{\phi}^{l/p'-1}|\nabla\phi|\label{29}\\
	\phi^{-1/p}|\Delta\phi|&=l(l-1)|\nabla\tilde{\phi}|^2\tilde{\phi}^{-l/p+l-2}+l(\Delta\tilde{\phi})\tilde{\phi}^{-l/p+l-1}\\
	&=l(l-1)|\nabla\tilde{\phi}|^2\tilde{\phi}^{l/p'-2}+l(\Delta\tilde{\phi})\tilde{\phi}^{l/p'-1}.\label{32}
	\end{align}
	The values appearing in the powers of \eqref{29} and \eqref{32} are all positive (note that $l>2p'$). So \eqref{Integral Check} is bounded. So, we get $\eqref{lemin}\les I(R)^{1/p}R^{n/p'}$.
	
	Next, we will check the estimate $\eqref{term1}\les -\eps R^{n/2}(\log R)^{-1}$.
	Considering Parseval's identity, 
	\begin{align}
	&\int_{\bR^n}(u_0+u_1)(2+\cos(a\cdot x))\phi_R(x)dx\\
	=&\int_{\bR^n}\what{(u_0+u_1)}(2+\cos(a\cdot x))\check{\phi_R}d\xi\\
	=&\int_{\bR^n}\what{(u_0+u_1)}(\xi-a)\check{\phi_R}d\xi+2\int_{\bR^n}\what{(u_0+u_1)}(\xi)\check{\phi_R}d\xi+\int_{\bR^n}\what{(u_0+u_1)}(\xi+a)\check{\phi_R}d\xi\\
	\ge&\int_{\bR^n}\what{(u_0+u_1)}(\xi+a)\check{\phi_R}d\xi\\
	\gtr&\int_{|\xi|<r}|\xi|^{-2/n}(\log(|\xi|^{-1}))^{-1}\check{\phi_R}d\xi\\
	=&\int_{|\xi|<rR}R^{n/2}|\xi|^{-n/2}(\log(R|\xi|^{-1}))^{-1}\check\phi(\xi)d\xi.\label{ini3}
	\end{align}
	The value \eqref{ini3} is equivalent to the value \eqref{ini2} with the case $\ga=0$. Therefore, by making the same argument as in Theorem\ref{b-u1}, we get $\eqref{term1}\les -\eps R^{n/2}(\log R)^{-1}$.
	\end{proof}
	
	By \eqref{lem2} and \eqref{lem3}, we can get
	\begin{align}
	0\le J(R)\les&-\eps R^{n/2}(\log R)^{-1}+CI(R)^{1/p}R^{n/p'}\\
	\le&-\eps R^{n/2}(\log R)^{-1}+C\left(-\eps\int_{\bR^n}(u_0+u_1)\phi_Rdx+CR^{n+2-2p'}\right)^{1/p}R^{n/p'}.
	\end{align}
	Here, Considering
	\begin{align}
	\int_{\bR^n}(u_0+u_1)\phi_Rdx=\int_{\bR^n}\what{(u_0+u_1)}\check\phi_Rdx>0\ \left(\text{by \eqref{ii} and $\what{(u_0+u_1)}>0$}\right),
	\end{align}
	we obtain
	\begin{align}
	0&\le-\eps R^{n/2}(\log R)^{-1}+CR^{(n+2-2p')/p+n/p'}\\
	\eps&\les R^{(2-2p')/p+n/2}\log R\le R^{(2-2p')/p+n/2+\tilde{\delta}}
	\end{align}
	for any $0<\tilde{\delta}\ll1$. By taking the limit $R\to \sqrt{T}$ and appropriately selecting $\tilde{\delta}$ in response to $\delta$, we get the lifespan 
	estimate \eqref{b-u3l}.
	\end{proof}

\section*{Acknowledgment}
I am deeply grateful to my supervisor, Prof. Mitsuru Sugimoto (Nagoya University), for his guidance throughout this work. I also thank Prof. Yuta Wakasugi (Hiroshima University) for helpful suggestions and for directing me to key references related to this work. I am grateful to Prof. Soichiro Suzuki (Chuo University) for providing the initial idea of  condition \eqref{inHg2}.
	

Graduate School of  Mathematics, Nagoya University, Furocho, Chikusaku, Nagoya 464-8602, Japan

Email: mitsuhiro.matsunaga.e6@math.nagoya-u.ac.jp
\end{document}